\documentclass[12pt,a4paper]{article}

\usepackage{amsbsy,amsfonts,amsmath,amssymb,amsthm}
\usepackage{array}
\usepackage{bm}
\usepackage{color}
\usepackage{enumerate}
\usepackage{mathrsfs}
\usepackage{latexsym}
\usepackage[colorlinks=true]{hyperref}
\hypersetup{urlcolor=blue, citecolor=blue, linkcolor=blue}

\definecolor{wineRed}{rgb}{0.7,0,0.3}
\definecolor{grandBleu}{rgb}{0,0,0.8}
\definecolor{darkGreen}{rgb}{0,0.4,0}
\definecolor{blueViolet}{rgb}{0.4,0,1.0}
\definecolor{bloodOrange}{rgb}{0.85,0.05,0}
\definecolor{mycolor}{rgb}{0.8,0,0.2}
\definecolor{}{rgb}{0.8,0,0.2}

\usepackage[textsize=small]{todonotes}
\usepackage{cite}

\usepackage{comment}

\DeclareMathAlphabet{\mathpzc}{OT1}{pzc}{m}{it}

\usepackage[textsize=small]{todonotes}
\numberwithin{equation}{section}

\theoremstyle{plain}
\newtheorem{theorem}{Theorem}[section]

\newtheorem{lemma}[theorem]{Lemma}

\theoremstyle{definition}
\newtheorem{definition}[theorem]{Definition}
\newtheorem{remark}{Remark}
\newtheorem{ex}{Example}

\def\Sgn{\mathop{\mathrm{Sgn}}\nolimits}
\def\ds{\displaystyle}

\begin{document}
\thispagestyle{plain}
\begin{center}
    \textbf{\Large Subdifferential decomposition of 1D-regularized total variation with nonhomogeneous coefficients}\footnotemark[1]
\end{center}
    \bigskip
\vspace{-0.5ex}
\begin{center}
    \textsc{Shodai Kubota}
    \\[1ex]
    {Department of Mathematics and Informatics, \\ Graduate School of Science and Engineering, Chiba University, \\ 1-33, Yayoi-cho, Inage-ku, 263-8522, Chiba, Japan}
    \\[0ex]
    ({\ttfamily skubota@chiba-u.jp})
\end{center}
\vspace{-1ex}

\footnotetext[1]{
\\[1ex]
}
\bigskip

\noindent
{\bf Abstract.} In this paper, we consider a convex function defined as a 1D-regularized total variation with nonhomogeneous coefficients, and prove the Main Theorem concerned with the decomposition of the subdifferential of this convex function to a weighted singular diffusion and a linear regular diffusion. The Main Theorem will be to enhance the previous regularity result for quasilinear equation with singularity, and moreover, it will be to provide some useful information in the advanced mathematical studies of grain boundary motion, based on KWC type energy.

\newpage

\section{Introduction}
    Let $\Omega := (-L, L) \subset \mathbb{R}$ be a one-dimensional spatial domain with a constant $0 < L < \infty$, and let us define $H := L^2(\Omega)$ and $V := H^1(\Omega)$. Let $ 0 \leq \alpha \in V $ and $ 0 < \beta \in V $ be fixed functions. 

In this paper, we consider the following convex function on $H$:
\begin{align}\label{Phi}
    \theta \in H \mapsto \Phi_{\alpha, \beta}(\theta) := V_\alpha(\theta) + W_\beta(\theta);
\end{align}
which is defined as a sum of two convex functions on $H$, defined as follows:
    \begin{align}\label{V}
    \displaystyle \theta \in H \mapsto V_\alpha(\theta) := 
\sup \left\{ \begin{array}{l|l}
        \displaystyle \int_\Omega \theta \partial_x \varphi \, dx, & \parbox{4cm}{$ \varphi \in V \cap C_\mathrm{c}(\Omega) $, such that $ |\varphi| \leq \alpha $ on $ \overline{\Omega} $}
    \end{array} \right\},
    \end{align}
and
\begin{align}\label{W}
    \displaystyle \theta \in H \mapsto W_\beta(\theta) := \left\{
\begin{array}{l}
\displaystyle \frac{1}{2}\int_{\Omega} \beta |\partial_x \theta|^2 dx, \, \mbox{ if } \theta \in V, \\[1ex]
\infty, \mbox{ otherwise.}
\end{array}
\right.
\end{align}
    The functional $ V_\alpha $, defined in \eqref{V}, is a kind of generalized total variation, so that the functional $ \Phi_{\alpha, \beta} $, defined in \eqref{Phi}, can be called a \emph{regularized total variation} with nonhomogeneous coefficients $ \alpha $ and $ \beta $.

On this basis, we set the goal to prove the following Main Theorem. 
\bigskip

\noindent
    \textbf{Main Theorem \textmd{(Decomposition of the subdifferential)}.} The subdifferential $ \partial \Phi_{\alpha, \beta} $ of the convex function $ \Phi_{\alpha, \beta} \subset H \times H $ is decomposed as follows:
\begin{align}\label{target}
    \partial \Phi_{\alpha, \beta} = \partial V_\alpha + \partial W_\beta \mbox{ in } H \times H,
\end{align}
i.e. $\partial \Phi_{\alpha, \beta}$ is represented as the sum the subdifferentials $ \partial V_\alpha \subset H \times H $ and $ \partial W_\beta \subset H \times H $ of the respective convex functions $ V_\alpha $ and $ W_\beta $. 
\bigskip

The equation \eqref{target} leads to the $ H^2 $-regularity of the following nonhomogeneous quasilinear equation with singularity:
\begin{equation}\label{elp01}
    \begin{cases}
        \displaystyle
        -\partial_x \left( \alpha(x) \frac{D \theta}{|D \theta|} +\beta(x) \partial_x \theta \right) = \theta^* \mbox{ with $ \theta^* \in H $,}
        \\[1.5ex]
        \mbox{subject to the zero-Neumann type boundary condition.}
    \end{cases}
\end{equation}
When the both $ \alpha $ and $ \beta $ are homogeneous (constants), we can obtain the $ H^2 $-regularity by using the mathematical method, developed in \cite{MR3144071}, which is based on the general theory of PDEs (e.g. \cite{MR0241822}). However, when $ \alpha $ and $ \beta $ are nonhomogeneous, the extra error terms brought by $ \alpha $ and $ \beta $ make it difficult to see $ \theta \in H^2(\Omega) $ in \eqref{elp01}, by referring to the existing method. Hence, it can be said that our Main Theorem will be to enhance the previous method of \cite{MR3144071}, and moreover, to report another variational approach based on the subdifferential. 

In the meantime, the Main Theorem is motivated by the mathematical analysis of grain boundary motion, studied in \cite{MR3362773,MR4177183}, and especially, the convex function $ \Phi_{\alpha, \beta} $ is based on the \emph{KWC energy}, proposed by Kobayashi--Warren--Carter \cite{MR1752970}. In this context, the variable $ \theta $ is the order parameter of crystalline orientation, and the nonhomogeneous coefficients $ \alpha $ and $ \beta $ are associated with another order parameter, such as the orientation order of grain in a polycrystal. In this light, our Main Theorem can be expected to provide useful information for some advanced problems that require smoothness of the system while including singularity, such as the optimal control problem governed by the KWC type model. 

The proof of Main Theorem is divided in three Sections. In the next Section 2, we prepare notations and mathematical theories as the preliminaries. Additionally, in Section 3, we prove an auxiliary lemma associated with the approximating approach to the Main Theorem. Based on these, the final Section 4 is devoted to the proof of our Main Theorem.

\section{Preliminaries} 

We begin by prescribing the assumptions and notations used throughout this paper. 
\medskip

\noindent
\underline{\textbf{\textit{Assumptions.}}} 
Throughout this paper, let $\Omega := (-L, L) \subset \mathbb{R}$ be a fixed spatial bounded domain with a constant $ 0 < L < \infty $, and let $\Gamma := \partial \Omega = \{-L, L\}$ be the boundary of $\Omega$. Also, let $ \partial_x $ be the distributional spatial differential. On this basis, we define  
\begin{align*} 
    H := L^2(\Omega), &\, H_\Gamma := \left\{ \begin{array}{l|l}
            \tilde{z} & \tilde{z} : \Gamma \longrightarrow \mathbb{R}
        \end{array} \right\} ~ (\sim \mathbb{R}^2), \mbox{ and }  V := H^1(\Omega) ~(\subset C(\overline{\Omega})).
\end{align*} 

Let $ \alpha \in V $ and $ \beta \in V $ be fixed functions, such that:
\begin{equation}\label{ass01}
    \min \alpha(\overline{\Omega}) \geq 0, \mbox{ and } \min \beta(\overline{\Omega}) > 0. 
\end{equation}
\medskip

\noindent
\underline{\textbf{\textit{Abstract notations.}}}
For an abstract Banach space $ X $, we denote by $ |\cdot|_{X} $ the norm of $ X $. Let $I_X : X \longrightarrow X $ be the identity map from $X$ onto $X$. In particular, when $ X $ is a Hilbert space, we denote by $ (\cdot,\cdot)_{X} $ the inner product of $ X $. 

For any subset $ A $ of a Banach space $ X $, let $ \chi_A : X \longrightarrow \{0, 1\} $ be the characteristic function of $ A $, i.e.:
    \begin{equation*}
        \chi_A: w \in X \mapsto \chi_A(w) := \begin{cases}
            1, \mbox{ if $ w \in A $,}
            \\[0.5ex]
            0, \mbox{ otherwise.}
        \end{cases}
    \end{equation*}
\medskip

\noindent
\underline{\textbf{\textit{Notations in convex analysis. (cf. \cite[Chapter II]{MR0348562})}}} 
    Let $ X $ be an abstract Hilbert space $ X $. For a proper, lower semi-continuous (l.s.c.), and convex function $ \Psi : X \longrightarrow (-\infty, \infty] $ on a Hilbert space $ X $, we denote by $ D(\Psi) $ the effective domain of $ \Psi $. Also, we denote by $\partial \Psi$ the subdifferential of $\Psi$. The subdifferential $ \partial \Psi $ corresponds to a weak differential of convex function $ \Psi $, and it is known as a maximal monotone graph in the product space $ X \times X $. The set $ D(\partial \Psi) := \bigl\{ z \in X \ |\ \partial \Psi(z) \neq \emptyset \bigr\} $ is called the domain of $ \partial \Psi $. We often use the notation ``$ [z_{0}, z_{0}^{*}] \in \partial \Psi $ in $ X \times X $\,'', to mean that ``$ z_{0}^{*} \in \partial \Psi(z_{0})$ in $ X $ for $ z_{0} \in D(\partial\Psi) $'', by identifying the operator $ \partial \Psi $ with its graph in $ X \times X $.
\medskip

\begin{ex}[Examples of the subdifferential]\label{exConvex}
   For any $ \varepsilon \geq 0 $, let $ f^\varepsilon : \mathbb{R} \longrightarrow [0, \infty) $ be a continuous and convex function, defined as follows:
    \begin{equation}\label{f_eps}
        f^\varepsilon : y \in \mathbb{R} \mapsto f^\varepsilon(y) := \sqrt{\varepsilon^2 +|y|^2} \in [0, \infty).
    \end{equation}

    When $ \varepsilon > 0 $, $ f^\varepsilon \in C^\infty(\mathbb{R}) $, and hence the subdifferential $ \partial f^\varepsilon \subset \mathbb{R} \times \mathbb{R} $ coincides with the single-valued function of the standard differential $ (f^\varepsilon)' \in L^\infty(\mathbb{R}) $, i.e.: 
    \begin{equation*}
        D(\partial f^\varepsilon) = \mathbb{R}, \mbox{ and } \partial f^\varepsilon(y) = (f^\varepsilon)'(y) = \frac{y}{\sqrt{\varepsilon^2 +|y|^2}}, ~ \mbox{for any $  y \in \mathbb{R} $.}
    \end{equation*}

    Meanwhile, when $ \varepsilon = 0 $, the corresponding function $ f^0 $ coincides with the function of absolute value $ |\cdot| : \mathbb{R} \longrightarrow [0, \infty) $. Hence, the subdifferential $ \partial f^0 $  of this case coincides with the set-valued signal function $ \Sgn: \mathbb{R} \longrightarrow 2^{\mathbb{R}} $, which is defined as follows:
\begin{align}\label{Sgn^d}
    & \xi \in \mathbb{R} \mapsto \Sgn(\xi) := \left\{ \begin{array}{ll}
            \multicolumn{2}{l}{
                \ds \frac{\xi}{|\xi|}, \mbox{ if $ \xi \ne 0 $,}}
                    \\[3ex]
            [-1, 1], & \mbox{otherwise,}
        \end{array} \right.
    \end{align}
i.e.:
\begin{equation*}
    D(\partial f^0) = D(\partial |\cdot|) = \mathbb{R}, \mbox{ and } \partial f^0(y) = \partial |{}\cdot{}|(y) = \Sgn(y), \mbox{ for any $ y \in \mathbb{R} $.}
\end{equation*}
\end{ex}
\medskip

Next, we mention about a notion of functional convergence, known as ``Mosco-convergence''. 
 
\begin{definition}[Mosco-convergence: cf. \cite{MR0298508}]\label{Def.Mosco}
    Let $ X $ be an abstract Hilbert space. Let $ \Psi : X \longrightarrow (-\infty, \infty] $ be a proper, l.s.c., and convex function, and let $ \{ \Psi_n \}_{n = 1}^\infty $ be a sequence of proper, l.s.c., and convex functions $ \Psi_n : X \longrightarrow (-\infty, \infty] $, $ n = 1, 2, 3, \dots $. Then, it is said that $ \Psi_n \to \Psi $ on $ X $, in the sense of Mosco, as $ n \to \infty $, iff. the following two conditions are fulfilled.
\begin{description}
    \item[(M1) The condition of lower-bound:]$ \ds \varliminf_{n \to \infty} \Psi_n(\check{w}_n) \geq \Psi(\check{w}) $, if $ \check{w} \in X $, $ \{ \check{w}_n  \}_{n = 1}^\infty \subset X $, and $ \check{w}_n \to \check{w} $ weakly in $ X $, as $ n \to \infty $. 
    \item[(M2) The condition of optimality:]for any $ \hat{w} \in D(\Psi) $, there exists a sequence $ \{ \hat{w}_n \}_{n = 1}^\infty  \subset X $ such that $ \hat{w}_n \to \hat{w} $ in $ X $ and $ \Psi_n(\hat{w}_n) \to \Psi(\hat{w}) $, as $ n \to \infty $.
\end{description}
\end{definition}

\begin{remark}\label{Rem.MG}
    Let $ X $, $ \Psi $, and $ \{ \Psi_n \}_{n = 1}^\infty $ be as in Definition~\ref{Def.Mosco}. Then, the following facts hold.
\begin{description}
    \item[(Fact\,1)](cf. \cite[Theorem 3.66]{MR0773850}) Let us assume that
    \begin{equation*}
    \Psi_n \to \Psi \mbox{ on $ X $, in the sense of  Mosco, as $ n \to \infty $,}
    \vspace{-1ex}
\end{equation*}
and
\begin{equation*}
\left\{ ~ \parbox{10cm}{
$ [w, w^*] \in X \times X $, ~ $ [w_n, w_n^*] \in \partial \Psi_n $ in $ X \times X $, $ n \in \mathbb{N} $,
\\[1ex]
$ w_n \to w $ in $ X $, and $ w_n^* \to w^* $ weakly in $ X $, as $ n \to \infty $.
} \right.
\end{equation*}
Then, it holds that:
\begin{equation*}
[w, w^*] \in \partial \Psi \mbox{ in $ X \times X $, and } \Psi_n(w_n) \to \Psi(w) \mbox{, as $ n \to \infty $.}
\end{equation*}
    \item[(Fact\,2)](cf. \cite[Lemma 4.1]{MR3661429} and \cite[Appendix]{MR2096945}) Let $ N \in \mathbb{N} $ denote a constant of dimension, and let $  S \subset \mathbb{R}^N $ be a bounded open set. Then, under the assumptions and notations as in (Fact\,1), a sequence $ \{ \widehat{\Psi}_n^S \}_{n = 1}^\infty $ of proper, l.s.c., and convex functions on $ L^2(S; X) $, defined as:
        \begin{equation*}
            z \in L^2(S; X) \mapsto \widehat{\Psi}_n^S(z) := \left\{ \begin{array}{ll}
                    \multicolumn{2}{l}{\ds \int_S \Psi_n(z(y)) \, dt,}
                    \\[1ex]
                    & \mbox{ if $ \Psi_n(z) \in L^1(S) $,}
                    \\[1.0ex]
                    \infty, & \mbox{ otherwise,}
                \end{array} \right. \mbox{for $ n = 1, 2, 3, \dots $;}
        \end{equation*}
        converges to a proper, l.s.c., and convex function $ \widehat{\Psi}^S $ on $ L^2(S; X) $, defined as:
        \begin{equation*}
            z \in L^2(S; X) \mapsto \widehat{\Psi}^S(z) := \left\{ \begin{array}{ll}
                    \multicolumn{2}{l}{\ds \int_S \Psi(z(y)) \, dt, \mbox{ if $ \Psi(z) \in L^1(S) $,}}
                    \\[2ex]
                    \infty, & \mbox{ otherwise;}
                \end{array} \right. 
        \end{equation*}
        on $ L^2(S; X) $, in the sense of Mosco, as $ n \to \infty $. 
\end{description}
\end{remark}
\begin{ex}[Example of Mosco-convergence]\label{Rem.ExMG}
    Let $ \{ f^\varepsilon \}_{\varepsilon \geq 0} \subset C(\mathbb{R}) $ be the sequence of nonexpansive convex functions, as in \eqref{f_eps}. Then, for any $\varepsilon_0 \geq 0, f^\varepsilon \to f^{\varepsilon_0}, \mbox{ uniformly on } \mathbb{R}, \mbox{ as } \varepsilon \to \varepsilon_0$, so that:
    \begin{equation*}
        f^\varepsilon \to f^{\varepsilon_0} \mbox{ on $ \mathbb{R} $, in the sense of Mosco, as $ \varepsilon \to \varepsilon_0 $.}
    \end{equation*}
\end{ex}

\noindent
\underline{\textbf{\textit{Basic and specific notations.}}} 
~For arbitrary $ r_0 $, $ s_0 \in [-\infty, \infty]$, we define:
\begin{equation*}
r_0 \vee s_0 := \max\{r_0, s_0 \}\ \mbox{and}\ r_0 \wedge s_0 := \min\{r_0, s_0 \},
\end{equation*}
and in particular, we set:
\begin{equation*}
    [r]^+ := r \vee 0 \ \mbox{and}\ [r]^- :=  -(r \wedge 0), \mbox{ for any $ r \in \mathbb{R} $.}
\end{equation*}

Finally, we remark on the specific functionals $ V_\alpha : H \longrightarrow [0, \infty] $, $ W_\beta : H \longrightarrow [0, \infty] $, and $\Phi_{\alpha, \beta} : H \longrightarrow [0, \infty]$, that are defined in \eqref{V}, \eqref{W}, and \eqref{Phi}, respectively. 

\begin{remark}(cf. \cite{MR1857292,MR1736243})\label{Rem.V_alp}
The functional $ V_\alpha $ coincides with the so-called \emph{lower semi-continuous envelope} of the following convex function:
\begin{equation*}
    \theta \in W^{1, 1}(\Omega) \mapsto \widetilde{V}_{\alpha}(\theta) := \int_\Omega \alpha |\partial_x \theta| \, dx \in [0, \infty), 
\end{equation*}
more precisely,
%
\begin{align}\label{V_another}
    & V_\alpha(\theta) = 
\inf \left\{ \begin{array}{l|l}
    \displaystyle \varliminf_{i \to \infty} \widetilde{V}_\alpha(\tilde{\vartheta}_i) & \parbox{4cm}{$ \{ \tilde{\vartheta}_i \}_{i = 1}^\infty \subset W^{1, 1}(\Omega) $, and $ \tilde{\vartheta}_i \to \theta $ in $ H $, as $ i \to \infty $}
\end{array} \right\}, 
    \\
    & \hspace{20ex}\mbox{for any $ \theta \in H $.}
    \nonumber
\end{align}
In the light of \eqref{V} and \eqref{V_another}, we can verify the following facts. 
\begin{description}
    \item[(Fact\,3)]$ V_\alpha $ is a proper, l.s.c., and convex function on $ H $, such that:
        \begin{itemize}
            \item the restriction $ V_\alpha|_{W^{1, 1}(\Omega)} $ coincides with $ \widetilde{V}_\alpha $;
            \item $ D(V_\alpha) \supset BV(\Omega) $, and $ D(V_\alpha) = BV(\Omega) $ if $ \min\alpha(\overline{\Omega}) > 0 $. 
        \end{itemize}
    \item[(Fact\,4)]For any $ \theta \in D(V_\alpha) $, there exists $ \{ \vartheta_i \}_{i= 1}^\infty \subset W^{1, 1}(\Omega) $ such that $ \vartheta_i \to \theta $ in $ H $, and $ \widetilde{V}_{\alpha}(\vartheta_i) \to V_\alpha(\theta) $, as $ i  \to \infty $.
\end{description}
\end{remark}
\begin{remark}\label{Rem.W_bt}
    The functional $ W_\beta $ is a proper, l.s.c., and convex function on $ H $, such that $ D(W_\beta) = V $. Moreover, the subdifferential $ \partial W_\beta \subset H \times H $ is a single valued operator, such that
    \begin{equation*}
        [\theta, \theta^*] \in \partial W_\beta \mbox{ in $ H \times H $, iff. $ \beta \partial_x \theta \in H_0^1(\Omega) $, and $ \theta^* = -\partial_x (\beta \partial_x \theta) $ in $ H $.}
    \end{equation*}
\end{remark}
\begin{remark}\label{Rem.Phi_eps}
 Let us fix $\varepsilon \geq 0$ and let $\Phi_{\alpha, \beta}^\varepsilon$ be a function on $H$, defined as follows: 
 \begin{align}\label{phi.e}
\displaystyle \Phi_{\alpha, \beta}^\varepsilon(\theta) := \left\{
\begin{array}{l}
\displaystyle \int_\Omega \alpha \sqrt{\varepsilon^2 + |\partial_x \theta|^2} \, dx+\frac{1}{2}\int_{\Omega} \beta |\partial_x \theta|^2 dx, \, \mbox{ if } \theta \in V, \\[2ex]
\infty, \mbox{ otherwise.}
\end{array}
\right.
\end{align}
Under the assumption \eqref{ass01}, the functions $\Phi^\varepsilon_{\alpha, \beta}$, for $\varepsilon \geq 0$, are proper, l.s.c., and convex on $H$. Especially, when $\varepsilon = 0$, the corresponding functional $\Phi^0_{\alpha, \beta}$ coincides with the convex function $\Phi_{\alpha, \beta}$, defined in \eqref{Phi}. 
\end{remark}
\begin{remark}\label{lem.op.ap}
Let us fix any $\varepsilon > 0$, and let us define a map $\mathcal{A}^\varepsilon :D(\mathcal{A}^\varepsilon) \subset H \longrightarrow H$, by putting: 
\begin{equation*}
 D(\mathcal{A}^\varepsilon) := \left\{ \theta \in V ~\Bigl|~ \alpha (f^\varepsilon)'(\partial_x \theta) + \beta \partial_x \theta \in H_{0}^{1}(\Omega)\right\},
\end{equation*}
and
\begin{equation*}
 \theta \in D(\mathcal{A}^\varepsilon) \subset H \mapsto \mathcal{A}^\varepsilon \theta := -\partial_x \bigl( \alpha(f^\varepsilon)'(\partial_x \theta) + \beta \partial_x \theta\bigr).
\end{equation*}
Then, by applying the standard variational technique, we can observe that:
\begin{align*}
\mathcal{A}^\varepsilon = \partial \Phi_{\alpha, \beta}^\varepsilon \mbox{ in } H \times H.
\end{align*}
\end{remark}

\section{Auxiliary lemma}

In this Section, we prove an auxiliary lemma which is associated with the approximating approach to the Main Theorem. 

\begin{lemma}\label{Mosco.psi}
Let $\{\varepsilon_m \}_{m = 1}^\infty \subset (0, \infty)$ be arbitrary sequence such that $\varepsilon_m \to 0$ as $m \to \infty$. Then, for the sequence $\{ \Phi_{\alpha, \beta}^{\varepsilon_m}\}_{m=1}^\infty$, it holds that: 
\begin{align*}
 \Phi_{\alpha, \beta}^{\varepsilon_m} \to \Phi_{\alpha, \beta} \mbox{ on } H, \mbox{ in the sense of Mosco}, \mbox{ as } m \to \infty.
\end{align*}
\end{lemma}

\begin{proof}
 First, we show the lower-bound condition (M1) in Definition \ref{Def.Mosco}. 
Let $\theta \in H$ and $\{ \theta^m\}_{m=1}^\infty \subset H$ be such that: 
\begin{align}\label{keylem01}
 \theta^m \to \theta \mbox{ weakly in } H, \mbox{ as } m \to \infty.
\end{align}
Then, it is sufficient to consider only the case when $\varliminf_{m \to \infty} \Phi_{\alpha, \beta}^{\varepsilon_m}(\theta^m) < \infty$, since the other case is trivial. 
So, by taking a subsequence $\{ m_k\}_{k=1}^\infty \subset \{m\}$, one can say that: 
\begin{align}\label{keylem02}
 \varliminf_{m \to \infty}\Phi_{\alpha, \beta}^{\varepsilon_m}(\theta^m) = \lim_{k \to \infty}\Phi_{\alpha, \beta}^{\varepsilon_{m_k}}(\theta^{m_k}) < \infty.
\end{align} 
With \eqref{phi.e}, \eqref{keylem01}, and \eqref{keylem02} in mind, we further see that: 
\begin{align}\label{keylem03}
 \partial_x \theta^{m_k} &\, \to \partial_x \theta \mbox{ weakly in } H, \nonumber \\
 \mbox{ and } \sqrt{\beta}\partial_x \theta^{m_k} &\, \to \sqrt{\beta} \partial_x \theta \mbox{ weakly in } H, \mbox{ as } k \to \infty,
\end{align}
by taking more one subsequence if necessary.
In the light of \eqref{f_eps}, \eqref{keylem01}--\eqref{keylem03}, Remark \ref{Rem.W_bt}, weakly lower semi-continuity of $\Phi_{\alpha, \beta}$, the lower-bound condition can be verified (M1), as follows: 
\begin{align*}
 \ds \varliminf_{k \to \infty}\Phi_{\alpha, \beta}^{\varepsilon_{m_k}}(\theta^{m_k}) &\, \geq \varliminf_{k \to \infty}\Phi_{\alpha, \beta}(\theta^{m_k}) \geq \Phi_{\alpha, \beta}(\theta).\\
\end{align*}

Next, we show the optimality condition (M2) in Definition\ref{Def.Mosco}. 
Let us fix any $\theta \in D(\Phi_{\alpha, \beta}) ( = V )$, and let us take a sequence $\{ \varphi^k\}_{k=1}^\infty \subset C^\infty(\overline{\Omega})$ such that:
\begin{align}\label{keylem04}
\varphi^k \to \theta \mbox{ in } V, \mbox{ and in the pointwise sense, a.e. in } \Omega, \mbox{ as } k \to \infty.
\end{align} 
By \eqref{keylem04} and Lebesgue's dominated convergence theorem, we can configure a sequence $\{ m_k\}_{k=0}^\infty \subset \mathbb{N}$ such that $1=: m_0 < m_1 < m_2 < \cdots < m_k \uparrow \infty$, as $k \to \infty$, and for any $k \in \mathbb{N} \cup\{0\}$, 
\begin{align}\label{keylem05}
 \ds \sup_{m \geq m_k}\bigl|f^{\varepsilon_m} (\partial_x \varphi^{k}) - |\partial_x \varphi^k|\bigr|_{L^1(\Omega)} < \frac{1}{2^k(|\alpha|_{L^\infty(\Omega)} + 1)}.
\end{align}
Based on these, let us define: 
\begin{equation}\label{keylem06}
 \theta^m := \left\{
\begin{array}{l}
\varphi^k \mbox{ if } m_k \leq m < m_{k+1}, \mbox{ for } k \in \mathbb{N}, \\[1mm]
\varphi^1 \mbox{ if } 1 \leq m < m_1, 
\end{array}
\mbox{ for any } m \in \mathbb{N}. \right.
\end{equation}
Taking into account \eqref{keylem04}--\eqref{keylem06} and H\"{o}lder's inequality, we obtain that:
\begin{align*}
 \ds \bigl|&\, \Phi_{\alpha, \beta}^{\varepsilon_m}(\theta^m) - \Phi_{\alpha, \beta}(\theta)\bigr| \\
           &\, \leq \left| \int_\Omega \bigl(\alpha f^{\varepsilon_m}(\partial_x \theta^m) - \alpha |\partial_x \theta|\bigr) \, dx\right| + \frac{1}{2} \int_\Omega \beta\bigl| |\partial_x \theta^m|^2 - |\partial_x \theta|^2 \bigr| \, dx\\
           &\, \leq |\alpha|_{L^\infty(\Omega)}\left( \int_\Omega \sup_{m \geq m_k}\bigl| f^{\varepsilon_m}(\partial_x \varphi^k) - |\partial_x \varphi^k|\bigr| dx + \int_{\Omega}\bigl| |\partial_x \varphi^k|  - |\partial_x \theta|\bigr| \, dx\right)\\
           &\, \qquad + \frac{|\beta|_{L^\infty(\Omega)}}{2}|\varphi^k - \theta|_V  \left( \int_\Omega 2(|\partial_x \varphi^k|^2 + |\partial_x \theta|^2)\, dx\right)^{\frac{1}{2}}  \\
           &\, \leq \frac{1}{2^{k}} + |\varphi^k - \theta|_V \cdot 
           \\
           &\, \qquad \cdot \left(\sqrt{2L}|\alpha|_{L^\infty(\Omega)} + \frac{|\beta|_{L^\infty(\Omega)}}{2} \left( \int_\Omega 2(|\partial_x \varphi^k|^2 + |\partial_x \theta|^2)\, dx\right)^{\frac{1}{2}}\right),  \\
           &\, \mbox{ for any } k \in \mathbb{N} \cup \{ 0\} \mbox{ and any } m \geq m_k,
\end{align*}
and therefore, 
\begin{align*}
 \Phi_{\alpha, \beta}^{\varepsilon_m}(\theta^m) \to \Phi_{\alpha, \beta}(\theta), \mbox{ as } m \to \infty. 
\end{align*}  

Thus, we conclude this lemma.
\end{proof}

\section{Proof of Main Theorem} 

In this Section, we give the proof of Main Theorem. 
Let us define a set-valued map $\mathcal{A}^0 :D(\mathcal{A}^0) \subset H \longrightarrow 2^H$, by putting:
\begin{equation}\label{A.0}
 D(\mathcal{A}^0) := \left\{\begin{array}{l|ll} \theta \in V & \begin{array}{lll}\multicolumn{2}{l}{\mbox{ there exists } \varpi^*\in L^\infty(\Omega) \mbox{ such that } } 
\\[0.25ex]
& \bullet ~ \varpi^* \in \Sgn(\partial_x \theta) \mbox{ a.e. in } \Omega
\\[0.25ex]
& \bullet ~ \alpha \varpi^* + \beta \partial_x \theta \in H_0^1(\Omega)  \end{array}
\end{array}\right\},
\end{equation}
and
\begin{align}\label{A.0f}
 \theta &\, \in D(\mathcal{A}^0) \subset H \nonumber
 \\
   &\, \mapsto \mathcal{A}^0 \theta := \left\{\begin{array}{l|l} \theta^* \in H & \begin{array}{lll}\multicolumn{2}{l}{ \theta^* = -\partial_x \bigl( \alpha\varpi^* + \beta \partial_x \theta\bigr) \mbox{ in } H, }
   \\[0.25ex]
   & \mbox{for some } \varpi^*\in L^\infty(\Omega), \mbox{ satisfying } 
   \\[0.25ex]
   & \varpi^*\in \Sgn(\partial_x \theta) \mbox{ a.e. in } \Omega \end{array}
 \end{array} \right\}.
\end{align}
\medskip
 
We prove Main Theorem in accordance with the following two Steps.
    \begin{description}
        \item[\textbf{\boldmath Step\,$1$:}]$\mathcal{A}^0 = \partial \Phi_{\alpha, \beta}$ in $H \times H$.
        \item[\textbf{\boldmath Step\,$2$:}]$\partial \Phi_{\alpha, \beta} = \partial V_\alpha + \partial W_\beta $ in $H \times H$. 
    \end{description}

\noindent
\textbf{\boldmath\underline{Verification of Step\,$1$.}}

First, we show $\mathcal{A}^0 \subset \Phi_{\alpha, \beta}$ in $H \times H$. 
Let us assume $\theta \in D(\mathcal{A}^0)$ and $\theta^* \in \mathcal{A}^0 \theta$. 
Then, by \eqref{A.0f}, there exists $\varpi^* \in L^\infty(\Omega)$ such that: 
\begin{align}\label{claim3-01}
 \varpi^* \in \Sgn(\partial_x \theta) \mbox{ a.e. in } \Omega \mbox{ and } \theta^* = -\partial_x(\alpha \varpi^* + \beta\partial_x \theta) \mbox{ in } H.
\end{align} 
From Remark \ref{Rem.V_alp}, \eqref{Sgn^d}, \eqref{claim3-01}, and Young's inequality, we can compute that: 
\begin{align*}
 \ds (\theta^*, \varphi - \theta)_H &\, = \bigl( -\partial_x \bigl( \alpha\varpi^* + \beta \partial_x \theta\bigr), \varphi - \theta \bigr)_H 
 \\
                      &\, = \int_\Omega \alpha\varpi^* \partial_x (\varphi - \theta)\, dx + \int_\Omega  \beta\, \partial_x \theta\, \partial_x (\varphi - \theta)\, dx 
                      \\
                      &\, \leq \int_{\Omega}\alpha\bigl( |\partial_x \varphi| - |\partial_x \theta|\bigr)\, dx + \frac{1}{2} \int_\Omega \beta(|\partial_x \varphi|^2 - |\partial_x \theta|^2 )\, dx 
                      \\
                      &\, = \Phi_{\alpha, \beta}(\varphi) - \Phi_{\alpha, \beta}(\theta), \mbox{ for any } \varphi \in V.                     
\end{align*}
This implies that: 
\begin{align*}
 \theta \in D(\partial \Phi_{\alpha, \beta}) \mbox{ and } \theta^* \in \partial\Phi_{\alpha, \beta}(\theta) \mbox{ in } H.
\end{align*}
Thus, the inclusion $ \mathcal{A}^0 \subset \partial \Phi_{\alpha, \beta} \mbox{ in } H \times H $ is verified.
\medskip

Next, we prove the equality $(\mathcal{A}^0 + I_H)H = H$.
Since, the inclusion $(\mathcal{A}^0 + I_H)H \subset H$ is trivial, it is sufficient to prove the converse inclusion. 

Let us take any $h \in H$.
Then, by Remark \ref{lem.op.ap} and Minty's theorem (cf. \cite[Theorem 2.2]{MR2582280}), we can configure a class of function $\{ \theta^\varepsilon \}_{\varepsilon> 0} \subset V$, by setting $
 \{ \theta^\varepsilon := (\mathcal{A}^\varepsilon + I_H)^{-1} h \}_{\varepsilon > 0} \mbox{ in } H, $
i.e.
\begin{align}\label{claim4-01}
 h - \theta^\varepsilon = \mathcal{A}^\varepsilon \theta^\varepsilon = \partial \Phi_{\alpha, \beta}^\varepsilon(\theta^\varepsilon) \mbox{ in } H, \mbox{ for any } \varepsilon > 0,
\end{align}
so that: 
\begin{align}\label{claim4-02}
 \int_\Omega \bigl( \alpha(f^\varepsilon)'(\partial_x \theta^\varepsilon) &\, + \beta \partial_x \theta^\varepsilon\bigr)\partial_x \varphi\, dx + \int_\Omega \theta^\varepsilon\varphi \, dx \nonumber
 \\
 &\, = \int_\Omega h\varphi\, dx, \mbox{ for any }\varphi \in V, \mbox{ and  any } \varepsilon > 0.
\end{align}
In the variational form \eqref{claim4-02}, let us put $\varphi = \theta^\varepsilon$. 
Then, with \eqref{f_eps} and Young's inequality in mind, we deduce that: 
\begin{align}\label{claim4-03}
 \ds \frac{1}{2}|\theta^\varepsilon|_H^2 + |\sqrt{\beta}\partial_x \theta^\varepsilon|_H^2 \leq \frac{1}{2}|h|_H^2, \mbox{ for any } \varepsilon > 0.
\end{align}
The above \eqref{claim4-03} enable us to take a function $\theta \in V$ and a sequence $\varepsilon_1 > \varepsilon_2 > \varepsilon_3 > \cdots > \varepsilon_m \downarrow 0$, as $m \to \infty$, such that:
\begin{align}\label{claim4-05}
 \theta^{\varepsilon_m} &\, \to \theta \mbox{ in } H, \mbox{ weakly in } V, \nonumber
 \\
  \mbox{ and } \sqrt{\beta} \partial_x \theta^{\varepsilon_m} &\, \to \sqrt{\beta} \partial_x \theta \mbox{ weakly in } H, \mbox{ as } m \to \infty.
\end{align}
In the light of Lemma \ref{Mosco.psi}, \eqref{claim4-01}, \eqref{claim4-05}, and (Fact\,1), it follows that: 
\begin{align}\label{claim4-06}
 h - \theta \in \partial \Phi_{\alpha, \beta}(\theta) \mbox{ in } H, \mbox{ and } \Phi_{\alpha, \beta}^{\varepsilon_m}(\theta^{\varepsilon_m}) \to \Phi_{\alpha, \beta}(\theta), \mbox{ as } m \to \infty.
\end{align}
Also, by Remark \ref{Rem.V_alp}, \eqref{f_eps}, \eqref{claim4-05}, \eqref{claim4-06}, and weakly lower semi-continuity of the norm $|\cdot|_H$, we can compute that: 
\begin{align}\label{claim4-07}
 \frac{1}{2}\int_{\Omega} \beta|\partial_x \theta|^2 \, dx & \leq \frac{1}{2}\varliminf_{m \to \infty}\int_{\Omega}\beta|\partial_x \theta^{\varepsilon_m}|^2 \, dx \leq \frac{1}{2}\varlimsup_{m \to \infty}\int_{\Omega} \beta|\partial_x \theta^{\varepsilon_m}|^2 \, dx \nonumber
 \\
 & \leq \lim_{m \to \infty}\Phi_{\alpha, \beta}^{\varepsilon_m}(\theta^{\varepsilon_m}) - \varliminf_{m \to \infty}\int_\Omega \alpha f^{\varepsilon_m}(\partial_x \theta^{\varepsilon_m})\, dx \nonumber
 \\
 & \leq \Phi_{\alpha, \beta}(\theta) - \int_\Omega \alpha|\partial_x \theta|\, dx = \frac{1}{2}\int_{\Omega}\beta |\partial_x \theta|^2\, dx.
\end{align}
Having in mind \eqref{claim4-05}, \eqref{claim4-07}, and the uniform convexity of $L^2$-based topologies, it is deduce that:
\begin{equation}\label{claim4-08}
\sqrt{\beta}\partial_x \theta^{\varepsilon_m} \to \sqrt{\beta}\partial_x \theta \mbox{ in } H, \mbox{ as } m \to \infty.
\end{equation}
Furthermore, by \eqref{ass01}, \eqref{claim4-05}, and \eqref{claim4-08}, we obtain that:
\begin{equation}\label{claim4-08-1}
 \theta^{\varepsilon_m} \to \theta \mbox{ in } V, \mbox{ and } \partial_x  \theta^{\varepsilon_m} \to \partial_x \theta \mbox{ in } H,
    \mbox{ as } m \to \infty.
\end{equation}

In the meantime, by Example \ref{exConvex}, $
 \ds |(f^{\varepsilon_m})'(\partial_x \theta^{\varepsilon_m})| \leq 1 \mbox{ a.e. in } \Omega, \mbox{ for any } m \in \mathbb{N},$
and one can say 
\begin{align}\label{claim4-10}
 (f^{\varepsilon_m})'(\partial_x \theta^{\varepsilon_m}) \to \varpi^* &\, \mbox{ weakly-$*$ in } L^\infty(\Omega), \mbox{ as } m \to \infty, \nonumber 
 \\
 &\, \mbox{ for some } \varpi^* \in L^\infty(\Omega), 
\end{align}
by taking a subsequence if necessary.

From \eqref{Sgn^d}, \eqref{claim4-08-1}, \eqref{claim4-10}, Example \ref{Rem.ExMG}, (Fact\,1), and \cite[Proposition 2.16]{MR0348562}, it is inferred that: 
\begin{align}\label{claim4-11}
 \varpi^* \in \Sgn(\partial_x \theta) \mbox{ a.e. in } \Omega. 
\end{align}
On account of \eqref{claim4-08}--\eqref{claim4-10}, letting $m \to \infty$ in \eqref{claim4-02} yields that: 
\begin{align}\label{claim4-12}
 \int_\Omega \bigl( \alpha\varpi^*  + \beta \partial_x \theta\bigr)\partial_x \varphi\, dx + \int_\Omega \theta \varphi \, dx = \int_\Omega h\varphi\, dx, \mbox{ for any }\varphi \in V.
\end{align}
In particular, putting $\varphi = \varphi_0 \in H^1_0(\Omega)$ in \eqref{claim4-12}, we have: 
\begin{align*}
 (h - \theta, \varphi_0)_H = \int_{\Omega}\bigl( \alpha \varpi^* + \beta \partial_x \theta, \partial_x \varphi_0\bigr) \, dx, \mbox{ for any } \varphi_0 \in H_0^1(\Omega), 
\end{align*} 
which implies: 
\begin{align}\label{claim4-12-2}
 -\partial_x \bigl( \alpha \varpi^* + \beta \partial_x \theta \bigr) = h - \theta \in H, \mbox{ in } \mathscr{D}'(\Omega).
\end{align}
In addition, we observe that: 
\begin{align}\label{claim4-12-3}
    \displaystyle &\, \Bigl( \alpha \varpi^* + \beta \partial_x \theta, \psi \Bigr)_{H_\Gamma}  = \left[ \rule{-1pt}{12pt} \bigl(\alpha(x) \varpi^*(x) + \beta(x)\partial_x \theta(x)\bigr)\psi(x) \right]_{-L}^{L} \nonumber 
\\
&\, ~ = \int_\Omega \partial_x \bigl( \bigl(\alpha \varpi^* + \beta \partial_x \theta  \bigr) [\psi]^{\mathrm{ex}}\bigr) \, dx \nonumber
\\
&\, ~ = - \int_\Omega(h - \theta)[\psi]^{\mathrm{ex}} \, dx + \int_\Omega \bigl(\alpha \varpi^* + \beta \partial_x \theta \bigr)\partial_x [\psi]^{\mathrm{ex}}\, dx \nonumber
\\
&\, ~ = 0, \mbox{ for any } \psi \in H_\Gamma \mbox{ with any extension } [\psi]^{\mathrm{ex}} \in V.
\end{align}
\eqref{claim4-12-2} and \eqref{claim4-12-3} lead to: 
\begin{align}\label{claim4-13}
 \alpha \varpi^* + \beta \partial_x \theta \in H_{0}^{1}(\Omega). 
\end{align} 
As a consequence of \eqref{A.0}, \eqref{A.0f}, \eqref{claim4-11}, and \eqref{claim4-13}, we obtain that: 
\begin{align*}
 (\mathcal{A}^0 + I_H) \theta = h \mbox{ in } H, \mbox{ i.e. } h \in (\mathcal{A}^0 + I_H)H,
\end{align*}
and we verify $H \subset (\mathcal{A}^0 + I_H)H$.

Finally, the inclusion $\mathcal{A}^0 \subset \partial \Phi_{\alpha, \beta}$ in $H \times H$, and the equality $(\mathcal{A}^0 + I_H)H = H$ enable us to apply Minty's theorem (cf. \cite[Theorem 2.2]{MR2582280}), 
and to verify that $\mathcal{A}^0$ is a maximal monotone. Moreover, the inclusion $\mathcal{A}^0 \subset \partial \Phi_{\alpha, \beta}$ and the maximality of $\mathcal{A}^0$ will lead to the coincidence $\mathcal{A}^0 = \partial \Phi_{\alpha, \beta}$ in $H \times H$. 

Thus we finish the proof of Step\,1.
\bigskip

\noindent
\textbf{\boldmath\underline{Verification of Step\,$2$.}}

By the general theory of the convex analysis \cite[Chapter 1]{MR1727362}, we immediately have $
\partial\Phi_{\alpha, \beta} \supset \partial V_\alpha + \partial W_\beta \mbox{ in } H \times H. $
So, we prove the converse inclusion:
\begin{align}\label{step2-00}
 \partial\Phi_{\alpha, \beta} \subset \partial V_\alpha + \partial W_\beta \mbox{ in } H \times H.
\end{align}

Let us take any $[\theta, \theta^*] \in \partial \Phi_{\alpha, \beta}$ in $H \times H$, and apply the result of previous Step\,1, to have a function $\varpi^* \in L^\infty(\Omega)$ as in \eqref{claim3-01}. 
On this basis, we verify this Step\,2, via the verifications of four Claims. 
\medskip

\noindent
\textbf{\boldmath\underline{Claim\,$\sharp 1)$.}}
$\theta \in H^2(\Omega)$ and $\partial_x \theta \in H_0^1(\Omega)$. 

For every $a \geq 0$ and $b > 0$, let $\rho_{(a, b)}: \mathbb{R} \longrightarrow 2^\mathbb{R}$ be a set-valued function, defined as:
\begin{align}\label{claim5-01}
 \rho_{(a, b)}(r) :=  a \Sgn(r) + b r \subset \mathbb{R}, \mbox{ for any } r \in \mathbb{R},
\end{align}
and let $\rho_{(a, b)}^{*}$ be the inverse of $\rho_{(a, b)}$. 
Then, as is easily checked from \eqref{Sgn^d} and \eqref{claim5-01}, 
\begin{align}\label{claim5-02}
 \displaystyle \rho_{(a, b)}^{*} &\, : r \in \mathbb{R} \mapsto \frac{[ r - a]^+ - [r + a]^-}{b} \in \mathbb{R},
\end{align}
i.e. $(\rho_{(a, b)})^{*}$ is a single-valued Lipschitz function, such that
\begin{align*}
\displaystyle 0 \leq (\rho_{(a, b)}^{*})' \leq \frac{1}{b} \mbox{ on } \mathbb{R}, \mbox{ for every } a \geq 0 \mbox{ and } b > 0.
\end{align*}
Here, from \eqref{claim5-01}, \eqref{claim5-02}, and Step\,1, we immediately see that: 
\begin{align}\label{claim5-03}
 \tilde{\theta} := &\, \rho_{(\alpha(\cdot), \beta(\cdot))}(\partial_x \theta) = \alpha \varpi^* + \beta \partial_x \theta \in H_0^1(\Omega), \mbox{ and } \theta^* = -\partial_x \tilde{\theta} \mbox{ in } H.
\end{align}
Therefore, having in mind \eqref{claim5-02} and \eqref{claim5-03}, and applying the generalized chain rule in BV-theory \cite[Theorem 3.99]{MR1857292}, it is inferred that:
\begin{align*}
 \displaystyle \partial_x \theta &\, = (\rho_{(\alpha(\cdot), \beta(\cdot))}^{*})( \tilde{\theta}) = \frac{[\tilde{\theta} -\alpha ]^+ - [\tilde{\theta} + \alpha]^-}{\beta} \in H_0^1(\Omega),
\end{align*} 
\vspace{-3.0ex}
\begin{align*}
 \displaystyle \partial_x^2 \theta = &\, \partial_x \left[\rule{-1pt}{12pt} \frac{[\tilde{\theta} -\alpha ]^+ - [\tilde{\theta} + \alpha]^-}{\beta} \right] 
 \\ 
 = &\, \frac{1}{\beta}\bigl[\rule{-1pt}{12pt} \chi_{(\alpha(\cdot), \infty)} + \chi_{(-\infty, -\alpha(\cdot))}\bigr](\tilde{\theta})\partial_x \tilde{\theta} - \frac{\partial_x \beta}{\beta^2}\bigl([\tilde{\theta} -\alpha ]^+ - [\tilde{\theta} + \alpha]^- \bigr) \in H.
\end{align*} 

Thus, Claim\,$\sharp 1)$ is verified.
\medskip

\noindent
\textbf{\boldmath\underline{Claim\,$\sharp 2)$.}}
$\beta \partial_x \theta \in H_0^1(\Omega)$ and $[\theta, -\partial_x(\beta \partial_x \theta)] \in \partial W_\beta$ in $H \times H$. 

This Claim\,$\sharp 2)$ is immediately observed from Claim$\sharp 1)$ and Remark \ref{Rem.W_bt}. 
\medskip

\noindent
\textbf{\boldmath\underline{Claim\,$\sharp 3)$.}}
$\alpha \varpi^* \in H_0^1(\Omega)$ and $[\theta, -\partial_x(\alpha \varpi^*)] \in \partial V_\alpha$ in $H \times H$.

By using \eqref{claim5-03}, Claim$\sharp 2)$, and the integration by part, we can observe that: 
\begin{align}\label{claim5-11}
 \alpha \varpi^* = \tilde{\theta} - \beta \partial_x \theta \in H_0^1(\Omega),
\end{align}
and
\begin{align}\label{claim5-12}
\displaystyle \int_\Omega &\, -\partial_x (\alpha \varpi^*)(\varphi - \theta)\, dx  = \int_\Omega \alpha \varpi^* \partial_x (\varphi - \theta)\, dx \nonumber 
\\
 &\, \leq \int_\Omega \alpha|\partial_x \varphi|\, dx - \int_\Omega \alpha |\partial_x \theta|\, dx, \mbox{ for any } \varphi \in W^{1,1}(\Omega). 
\end{align}

Next, let us take any $z \in D(V_\alpha)$, and invoke (Fact\,4) to take a sequence $\{\varphi_i \}_{i=1}^\infty \subset W^{1,1}(\Omega)$ such that: 
\begin{equation}\label{claim5-13} 
  \displaystyle \varphi_i \to z \mbox{ in } H, \mbox{ and } \widetilde{V}_\alpha(\varphi_i) \left( = \int_\Omega \alpha|\partial_x \varphi_i|\, dx \right) \to V_\alpha(z), 
 \mbox{ as } i \to \infty.
\end{equation}
Besides, putting $\varphi = \varphi_i$ in \eqref{claim5-12}, with $i \in \mathbb{N}$, and using \eqref{claim5-13}, we deduce that
\begin{align}\label{claim5-14}
 (-\partial_x(\alpha \varpi^*),&\, z - \theta )_H + V_\alpha(\theta) \nonumber 
 \\
 &\, = \lim_{i\to \infty} \int_\Omega -\partial_x (\alpha \varpi^*)(\varphi_i - \theta) \, dx + V_\alpha(\theta) \nonumber
 \\
 &\, \leq \lim_{i \to \infty} \widetilde{V}_{\alpha}(\varphi_i) = V_\alpha(z), \mbox{ for any } z \in D(V_\alpha).
\end{align}

\eqref{claim5-11} and \eqref{claim5-14} finish the verification of Claim\,$\sharp 3)$.
\medskip

\noindent
\textbf{\boldmath\underline{Claim\,$\sharp 4)$.}}
$\theta^* \in \partial V_\alpha(\theta) + \partial W_\beta(\theta)$ in $H$.

This Claim\,$\sharp 4)$ will be a straight forward consequence of \eqref{A.0f}, Step\,1, Claim\,$\sharp 1)$--Claim\,$\sharp 3)$, and the linearity of distributional differential:
\begin{align*}
 \theta^* &\, = -\partial_x(\alpha \varpi^* + \beta \partial_x \theta) = -\partial_x(\alpha \varpi^*) - \partial_x(\beta \partial_x \theta) \mbox{ in } \mathscr{D}'(\Omega).
\end{align*}

Claim\,$\sharp 1)$--Claim\,$\sharp 4)$ enable us to verify the inclusion \eqref{step2-00}, and to complete the proof of Main Theorem.
\qed

\section{Conclusion}
In this paper, the regularized total variation functional with nonhomogeneous coefficients is considered, and it is concluded that the subdifferential of this functional is decomposed to the sum of a weighted singular diffusion and a weighted linear diffusion. The result, stated in the Main Theorem, is to guarantee the $ H^2 $-regularity of the nonhomogeneously weighted quasilinear equations. The novelty of this work is in the point that the result is obtained by means of the approximation based on Mosco-convergence, and the generalized chain rule in BV-theory \cite[Theorem 3.99]{MR1857292}. Indeed, the mathematical method adopted here is different with the traditional approach based on the PDE-theory (cf. \cite{MR0241822}), and also, it would be a simple method to extend the result of the previous work \cite{MR3144071}.



    \textbf{Shodai Kubota}, ph.D student, Department of Mathematics and Informatics, Graduate School of Science and Engineering, Chiba University, 1--33, Yayoi-cho, Inage-ku, 263--8522, Chiba, Japan,\\ tel: +81 (0) 43 290 2665, email: skubota@chiba-u.jp, ORCID iD: 0000-0002-9495-0837

\end{document}